\documentclass[letterpaper,]{article}
\usepackage{etex}
\usepackage{graphicx}
\usepackage{amssymb}
\usepackage{amsmath}
\usepackage{amsthm}
\usepackage{pgf,tikz}
\usetikzlibrary{arrows}
\usepackage{dcpic}
\usepackage{pictex}
\usepackage{anysize}		
\marginsize{3 cm}{3 cm}{2.5 cm}{2.5 cm}
\newcommand{\Q}{\mathbb Q}
\newcommand{\QQ}{\overline{\Q}}
\newcommand{\Z}{\mathbb{Z}}
\newcommand{\PP}{\mathbb P}
\newcommand{\lra}{\longrightarrow}
\newcommand{\Gal}{\mathrm{Gal}}
\newcommand{\rd}{\sqrt[3]{2}}
\newcommand{\puntoblanco}[1]{
\draw [color=black] #1 circle (1.5pt);
\fill[color=white] #1 circle (1.3pt);
}
\newcommand{\puntonegro}[1]{
\fill [color=black] #1 circle (1.5pt);
}
\newcommand{\Fd}{\widehat{F}_2}
\newcommand{\cor}{\mathrm{core}}
\newcommand{\h}{{\widetilde{H}}}
\newcommand{\D}{\widetilde{\mathcal{H}}}
\newcommand{\zn}[1]{\Z/#1\Z}
\newcommand{\F}{\mathbb F}
\newcommand{\mat}[4]{\left( \begin{array}{cc}
#1 & #2 \\ #3 & #4
\end{array}\right)}
\newcommand{\HH}{\mathbb{H}}
\newcommand{\wt}[1]{\widetilde{#1}}

\newcommand{\mF}{\mathcal{F}}
\newcommand{\mE}{\mathcal{E}}
\newcommand{\mH}{\mathcal{H}}
\newcommand{\wmF}{\widetilde{\mathcal{F}}}
\newcommand{\wF}{\widetilde{F}}
\newcommand{\mB}{\mathcal{B}}
\newcommand{\Div}{\mathrm{Div}}
\newcommand{\Pic}{\mathrm{Pic}}
\newcommand{\Tor}{\mathrm{Tor}}
\newtheorem{lemma}{Lemma}

\title{An explicit quasiplatonic curve with non-abelian moduli field}		
\author{Mois\'es Herrad\'on Cueto\footnote{Supported by the program \textit{Posgrado de Excelencia Internacional en Matem\'aticas 2013-2014} at Universidad Aut\'onoma de Madrid}}		
	\date{}

\begin{document}

\maketitle

\paragraph*{Abstract}
We give an example of a regular dessin d'enfant whose field of moduli is not an abelian extension of the rational numbers, namely it is the field generated by a cubic root of 2. This answers a previous question. We also prove that the underlying curve has non-abelian field of moduli itself, giving an explicit example of a quasiplatonic curve with non-abelian field of moduli. In the last section, we note that two examples in previous literature can be used to find other examples of regular dessins d'enfants with non-abelian field of moduli.

\section{Introduction}
\label{intro}
Dessins d'enfants were first introduced by Alexander Grothendieck in his \textit{Esquisse d'un Programme} \cite{G} as pairs $(C,\beta)$, consisting of a smooth algebraic curve $C$ defined over the complex numbers, and a map $\beta:C\lra \PP^1(\mathbb C)=\PP^1$ unramified outside of the set $\{0,1,\infty\}$. Such maps are called Belyi maps. Belyi's theorem states that such an algebraic curve $C$ has a Belyi map if and only if it can be defined over $\QQ$. In this case, the Belyi map can also be defined over the algebraic numbers. This gives an action of the group $G_\Q=\Gal(\QQ/\Q)$ on the set of dessins d'enfants via the action of $G_\Q$ on the coefficients of the defining equations of $C$ and $\beta$. Dessins d'enfants are also in bijection with graphs embedded in orientable surfaces by considering $\beta^{-1}([0,1])$, which is a bipartite graph embedded in $C$.

For a Belyi pair $(C,\beta)$, one can consider the subgroup of $G_\Q$ that maps it to an isomorphic pair. The fixed field of this subgroup is the field of moduli of the dessin, an important invariant. In this paper, we find two examples of regular dessins d'enfants with field of moduli $\Q(\sqrt[3]{2})$, thus providing explicit examples of regular dessins d'enfants with non-abelian fields of moduli, which is a question posed in \cite{CJSW}. Although no explicit examples are in the literature, the existence of regular dessins with non-abelian field of moduli follows from a theorem of Jarden \cite{J}, and they are known to exist even if one fixes the type of the dessin \cite{GJ}. We prove that the underlying curve of one of them actually has itself the same field of moduli, so we also provide an example of a quasiplatonic curve with non-abelian field of moduli. We give explicit equations for this example.

For the construction, we look at a known dessin d'enfant on the elliptic curve with equation $Y^2=X(X-1)(X-\sqrt[3]{2})$, which has field of moduli $\Q(\rd)$ and was studied in \cite{W}. Then, we look at its regular cover, which is a dessin of degree $2^6\cdot 3^2=576$ on a curve of genus 145, and we prove it has field of moduli $\Q(\rd)$ as well. We also find a different regular dessin with the same field of moduli, which has this one as a degree two covering, and has smaller genus, namely 61. We give explicit equations for this last dessin.

The result is explained in further detail in the author's master's thesis \cite{H}.
\section{The main example}
\label{sec:1}
We remind the reader that a dessin d'enfant can be seen in many different ways. First of all, starting with a Belyi pair $(C,\beta)$, one can take the preimage of the interval $[0,1]$ to obtain a graph on the surface $C$, which is usually drawn with black vertices in the preimage of $0$ and white vertices in the preimage of $1$. Throughout, we will follow this convention. Since $\beta$ is a ramified cover, dessins are also in correspondence with their monodromy actions, which are homomorphisms from $\pi_1(\PP^1\setminus\{0,1,\infty\})=F_2$ onto a transitive subgroup of a permutation group. Since these homomorphisms factor through a finite quotient, we may think of them as homomorphisms from the profinite completion of $\Fd$ instead. This has the advantage that $G_\Q$ acts by automorphisms on $\Fd$ (see \cite{GJ}). Finally, from a monodromy action one can consider the stabilizer of a point, yielding a bijection between monodromy actions and conjugacy classes of finite index subgroups of $\Fd$. We will freely go from one setting to another. Possible references for details on this correspondence include \cite{Gu} and \cite{GG}.

Let $(C_{\mH_0},\beta_{\mH_0})$ be the Belyi pair consisting of the curve with equation $Y^2=X(X-1)(X-\rd)$ and the Belyi function $\beta_{\mH_0}(X,Y)=1-(X^3-1)^2$. We will call the corresponding dessin d'enfant $\mH_0$, which is drawn in Figure \ref{fig:toro}. The details about this dessin can be found in \cite{W}.

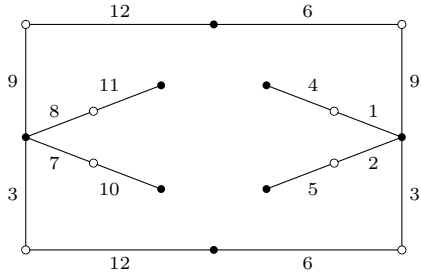
\begin{figure}[h!]
\begin{center}
\begin{tikzpicture}[line cap=round,line join=round,>=triangle 45,x=2.5cm,y=1.5cm]
\clip(-1.7,-1.5) rectangle (1.5,1.5);
\draw (1,1)--(1,-1)--(-1,-1)--(-1,1)--(1,1);
\draw (1,0)--(0.28,0.46);
\draw (1,0)--(0.28,-0.46);
\draw (-1,0)--(-0.28,0.46);
\draw (-1,0)--(-0.28,-0.46);
\begin{scriptsize}
\puntonegro{(1,0)} 
\puntoblanco{(1,1)} 
\puntonegro{(0,1)} 
\puntonegro{(0,-1)}
\puntonegro{(-1,0)} 
\puntoblanco{(-1,1)} 
\puntoblanco{(-1,-1)} 
\puntoblanco{(1,-1)} 
\puntoblanco{(0.64,0.23)}
\puntoblanco{(0.64,-0.23)} 
\puntoblanco{(-0.64,0.23)}
\puntoblanco{(-0.64,-0.23)}
\puntonegro{(0.28,0.46)}
\puntonegro{(0.28,-0.46)} 
\puntonegro{(-0.28,0.46)}
\puntonegro{(-0.28,-0.46)}
\draw (1,0.5) node[anchor=west] {9};
\draw (1,-0.5) node[anchor=west] {3};
\draw (-1,0.5) node[anchor=east] {9};
\draw (-1,-0.5) node[anchor=east] {3};
\draw (0.5,1) node[anchor=south] {$6$};
\draw (-0.5,1) node[anchor=south] {$12$};
\draw (0.5,-1) node[anchor=north] {6};
\draw (-0.5,-1) node[anchor=north] {12};
\draw (0.78,0.11) node[anchor=south west] {1};
\draw (0.78,-0.11) node[anchor=north west] {2};
\draw (-0.78,0.11) node[anchor=south east] {8};
\draw (-0.78,-0.11) node[anchor=north east] {7};
\draw(0.46,0.34) node[anchor=south west] {4};
\draw(0.46,-0.34) node[anchor=north west] {5};
\draw(-0.46,0.34) node[anchor=south east] {11};
\draw(-0.46,-0.34) node[anchor=north east] {10};
\end{scriptsize}
\end{tikzpicture}
\caption{The dessin d'enfant $\mH_0$. Opposite edges are identified to make a genus 1 surface.}\label{fig:toro}
\end{center}
\end{figure}

We take the usual generators of the fundamental group of $\PP^1\setminus \{0,1,\infty\}$, namely $x$, which corresponds to a positively oriented loop going around 0
; $y$, which corresponds to a positively oriented loop going around $1$; and $z=(xy)^{-1}$, which can be seen as going around $\infty$. Thus, the algebraic fundamental group is the free profinite group generated by $x$ and $y$. The correspondence between Belyi pairs and monodromy actions associates to the Belyi pair $(C_{\mH_0},\beta_{\mH_0})$ a transitive action of $\Fd$ on the edges of $\mH_0$. Furthermore, actions correspond to subgroups, by taking the stabilizer of an edge. We will call $H_0<\Fd$ the subgroup corresponding to $\mH_0$. In this case, the loops $x$ and $y$ act as the following permutations on the edges of $\mH_0$, with the numbering of Figure \ref{fig:toro}:
$$
\begin{array}{rrcl}
\rho_0:&x & \longmapsto & (1,2,3,7,8,9)(6,12) \\
&y & \longmapsto & (1,4)(2,5)(7,10)(8,11)(3,6,9,12)
\end{array}
$$
Let $\D_0$ be the regular cover of $\mH_0$, and let $\h_0=\cor_{\Fd} H_0=\ker \rho_0$\footnote{If $H<G$, we define $\cor_GH=\bigcap_{g\in G} gHg^{-1}$.} be the corresponding subgroup. In order to better understand $\D_0$, we will look at a normal subgroup of $\Fd$ that has index 18 and contains $\h_0$. Let $\mF$ be the dessin d'enfant on $\PP^1$ given by the map $\beta_{\mF}:X\longmapsto 1-(X^3-1)^2$ (drawn in Figure \ref{fig:P1}). Since $\beta_{\mH_0}$ factors through $\beta_{\mF}$, $\mF$ is an intermediate dessin of $\mH_0$. $\mF$ has a monodromy action given by $\rho_F$ (also in Figure \ref{fig:P1}). Its regular cover, $\wmF$, is an intermediate dessin of $\D_0$.

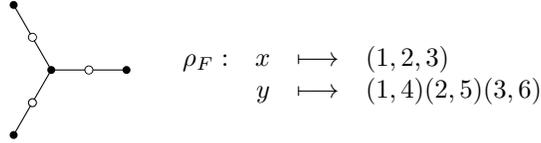
\begin{figure}[h!]
\begin{center}
\begin{tabular}{cc}
\begin{tikzpicture}[line cap=round,line join=round,>=triangle 45,x=0.5cm,y=0.5cm]
\clip(-2,-2) rectangle (2.3,2);
\draw (-1,-1.73)-- (0,0);
\draw (0,0)-- (2,0);
\draw (0,0)-- (-1,1.73);
\begin{scriptsize}
\fill [color=black] (0,0) circle (1.5pt);
\draw [color=black] (1,0) circle (1.5pt);
\fill [color=white] (1,0) circle (1.3pt);
\fill [color=black] (2,0) circle (1.5pt);
\draw [color=black] (-0.5,0.87) circle (1.5pt);
\fill [color=white] (-0.5,0.87) circle (1.3pt);
\fill [color=white] (-0.5,-0.87) circle (1.3pt);
\draw [color=black] (-0.5,-0.87) circle (1.5pt);
\fill [color=black] (-1,1.73) circle (1.5pt);
\fill [color=black] (-1,-1.73) circle (1.5pt);
\end{scriptsize}
\end{tikzpicture}
&
$
\begin{array}{rrcl}
\rho_F:&x & \longmapsto & (1,2,3) \\
&y & \longmapsto & (1,4)(2,5)(3,6) \\
\\
\\
\\
\\
\end{array}
$
\end{tabular}
\vspace{-0.7 cm}
\caption{The dessin $\mF$ and its monodromy homomorphism.}\label{fig:P1}
\end{center}
\end{figure}
\vspace{-0.4 cm}
Let us call $\wF$ the subgroup corresponding to the regular cover of this dessin. The group $\rho_0(\wF)$ is the subgroup of $\rho_0(\Fd)$ preserving the partition $\{ \{1,7\},\{2,8\},$ $\{3,9\},\{4,10\},\{5,11\},\{6,12\}\}$. One checks directly that $\wF$ is the normal subgroup generated by $\{x^3,y^2,[x,x^y]\}$\footnote{From now on, when $a,b$ are elements of a group, we denote $a^b=b^{-1}ab$.}.

The dessin $\mH_0$ has two conjugate dessins, which we will call $\mH_1$ and $\mH_2$, obtained when $\sqrt[3]{2}$ is replaced in the equation by its Galois conjugates. Thus, $\mH_1$ is a dessin on the curve with equation $Y^2=X(X-1)(X-\xi\sqrt[3]{2})$, where $\xi=\frac{-1+i\sqrt{3}}{2}$ is a cubic root of unity, and $\mH_2$ is a dessin on the curve with equation $Y^2=X(X-1)(X-\xi^2\sqrt[3]{2})$. They both have Belyi maps given by the expression $(X,Y)\longmapsto 1-(X^3-1)^2$, which look like Figure \ref{fig:toros}. One can obtain these pictures as degree 2 coverings of $\mF$, using the fact that a ramified covering of the sphere of degree 2 is uniquely determined by its set of ramification values. Note that in all three cases the degree 2 map to $\PP^1$ is given by the quotient by the $180^\circ$ rotation around the center of the rectangle (which is the automorphism with equation $(X,Y)\mapsto (X,-Y)$).
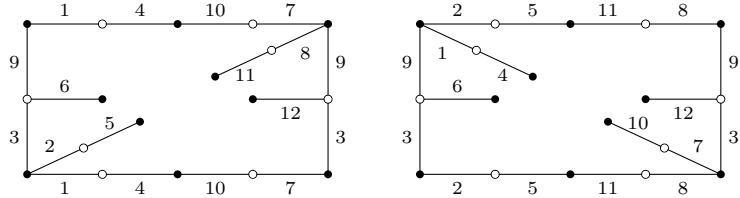
\begin{figure}[h!]
\begin{center}
\begin{tabular}{cc}
\begin{tikzpicture}[line cap=round,line join=round,>=triangle 45,x=1cm,y=1cm]
\clip(-2.4,-1.3) rectangle (2.4,1.3);
\draw (-2,-1)--(-2,1)--(2,1)--(2,-1)--(-2,-1);
\draw (-2,0)--(-1,0);
\draw (2,0) --(1,0);
\draw (-2,-1) --(-0.5,-0.3);
\draw (2,1) --(0.5,0.3);
\begin{scriptsize}
\puntoblanco{(-2,0)}
\puntoblanco{(2,0)}
\puntonegro{(-1,0)} 
\puntonegro{(1,0)} 
\puntonegro{(-2,1)}
\puntonegro{(-2,-1)}
\puntonegro{(2,1)}
\puntonegro{(2,-1)} 
\puntoblanco{(1,1)}
\puntoblanco{(-1,1)} 
\puntoblanco{(1,-1)} 
\puntoblanco{(-1,-1)}
\puntonegro{(0,1)}
\puntonegro{(0,-1)}
\puntoblanco{(1.25,0.65)} 
\puntoblanco{(-1.25,-0.65)}
\puntonegro{(0.5,0.3)} 
\puntonegro{(-0.5,-0.3)}
\draw (-1.5,1) node[anchor=south]{1};
\draw (-0.5,1) node[anchor=south]{4};
\draw (0.5,1) node[anchor=south]{10};
\draw (1.5,1) node[anchor=south]{7};
\draw (-1.5,-1) node[anchor=north]{1};
\draw (-0.5,-1) node[anchor=north]{4};
\draw (0.5,-1) node[anchor=north]{10};
\draw (1.5,-1) node[anchor=north]{7};
\draw (2,0.5) node[anchor=west]{9};
\draw (2,-0.5) node[anchor=west]{3};
\draw (-2,0.5) node[anchor=east]{9};
\draw (-2,-0.5) node[anchor=east]{3};
\draw (-1.7,-0.8) node[anchor=south]{2};
\draw (1.7,0.8) node[anchor=north]{8};
\draw (-0.9,-0.5) node[anchor=south]{5};
\draw (0.9,0.5) node[anchor=north]{11};
\draw (1.5,0) node[anchor=north]{12};
\draw (-1.5,0) node[anchor=south]{6};
\end{scriptsize}
\end{tikzpicture}
&
\begin{tikzpicture}[line cap=round,line join=round,>=triangle 45,x=1cm,y=1cm]
\clip(-2.4,-1.3) rectangle (2.4,1.3);
\draw (-2,-1)--(-2,1)--(2,1)--(2,-1)--(-2,-1);
\draw (-2,0)--(-1,0);
\draw (2,0) --(1,0);
\draw (-2,1) --(-0.5,0.3);
\draw (2,-1) --(0.5,-0.3);
\begin{scriptsize}
\puntoblanco{(-2,0)}
\puntoblanco{(2,0)}
\puntonegro{(-1,0)} 
\puntonegro{(1,0)} 
\puntonegro{(-2,1)}
\puntonegro{(-2,-1)}
\puntonegro{(2,1)}
\puntonegro{(2,-1)} 
\puntoblanco{(1,1)}
\puntoblanco{(-1,1)} 
\puntoblanco{(1,-1)} 
\puntoblanco{(-1,-1)}
\puntonegro{(0,1)}
\puntonegro{(0,-1)}
\puntoblanco{(1.25,-0.65)} 
\puntoblanco{(-1.25,0.65)}
\puntonegro{(0.5,-0.3)} 
\puntonegro{(-0.5,0.3)}
\draw (-1.5,1) node[anchor=south]{2};
\draw (-0.5,1) node[anchor=south]{5};
\draw (0.5,1) node[anchor=south]{11};
\draw (1.5,1) node[anchor=south]{8};
\draw (-1.5,-1) node[anchor=north]{2};
\draw (-0.5,-1) node[anchor=north]{5};
\draw (0.5,-1) node[anchor=north]{11};
\draw (1.5,-1) node[anchor=north]{8};
\draw (2,0.5) node[anchor=west]{9};
\draw (2,-0.5) node[anchor=west]{3};
\draw (-2,-0.5) node[anchor=east]{3};
\draw (-2,0.5) node[anchor=east]{9};
\draw (-1.7,0.8) node[anchor=north]{1};
\draw (1.7,-0.8) node[anchor=south]{7};
\draw (-0.9,0.5) node[anchor=north]{4};
\draw (0.9,-0.5) node[anchor=south]{10};
\draw (1.5,0) node[anchor=north]{12};
\draw (-1.5,0) node[anchor=south]{6};
\end{scriptsize}
\end{tikzpicture}
\end{tabular}
\caption{To the left, the dessin d'enfant $\mH_1$, and to the right, the dessin $\mH_2$.}\label{fig:toros}
\end{center}
\end{figure}
These dessins d'enfants have monodromy actions corresponding to the following homomorphisms from $\Fd$ to $S_{12}$:
$$
\begin{array}{rrcl}
\rho_1:&x & \longmapsto & (1,2,3,7,8,9)(4,10) \\
&y & \longmapsto & (1,4)(2,5)(7,10)(8,11)(3,6,9,12)\\
\rho_2:&x & \longmapsto & (1,2,3,7,8,9)(5,11) \\
&y & \longmapsto & (1,4)(2,5)(7,10)(8,11)(3,6,9,12)
\end{array}
$$
Let $H_1$ and $H_2$ be the subgroups of $\Fd$ corresponding to $\mH_1$ and $\mH_2$, and let $\h_1$ and $\h_2$ be the subgroups corresponding to their regular covers $\wt{\mH}_1$ and $\wt{\mH}_2$. Note that $\wF$ contains $\h_1$ and $\h_2$, since $\mF$ is an intermediate dessin of both $\mH_1$ and $\mH_2$.

The Galois action preserves regular covers, since $G_{\Q}$ acts as automorphisms of $\Fd$ (see \cite{GJ}). Therefore, if for $\sigma\in G_\Q$, $\sigma(\rd)=\xi\rd$, then $\h_0^\sigma=\h_1$, where the superscript denotes the Galois action. Let us prove that $\h_0\neq \h_1$. One just needs to check that $g=x^3y^2(x^3y^2)^x(x^3y^2)^{x^2}$ belongs to $\h_0$ but not to $\h_1$. Indeed, $\rho_0(g)=1$, but $\rho_1(g)=(4,10)(6,12)$. Thus, the field of moduli of $\D_0$ is $\Q(\rd)$.

\subsection{A different, smaller dessin}
We have two goals in this section. The first one will be to prove that the center of $\Fd/\h_0$ has order 2, which lifts to a subgroup $E_0\supset \h_0$. Via $\rho_0$, the generator of the center is mapped to the permutation $(1,7)(2,8)(3,9)(4,10)(5,11)(6,12)$. Since the center is a characteristic subgroup, the Galois group will map $E_0$ to subgroups $E_1\supset \h_1$ and $E_2\supset \h_2$, such that $E_i/\h_i$ are also the centers of $\Fd/\h_i$. We will have to prove that the $E_i$'s are also distinct, that is, to rule out the possibility that $E_i=\h_0\h_1\h_2$. This implies that the $E_i$'s correspond to a new Galois orbit of regular dessins $\mE_i$, which also have field of moduli $\Q(\sqrt[3]{2})$. The second goal is to learn more about $\Fd/E_0$, which will enable us to give equations for $\mE_0$ later. One can alternatively use a computer algebra system to find the center of $\Fd/\h_0$, which is also its unique normal subgroup of order 2, and to prove that the $E_i$'s are indeed distinct.

Let $K$ be the quotient $\wF/\h_0$. Using that $\wF$ is the normal subgroup generated by $\{x^3,y^2,[x,x^y]\}$, one checks that $\rho_0$ induces an isomorphism between $\wF/\h_0=K$ and the subgroup of permutations which are even and fix the partition $\{\{i,i+6\}:i=1,\ldots ,6\}$. Therefore, $K \cong (\zn{2})^5$. We will use the notation $a_i=(i,i+6)$ and use additive notation for this group. $K$ has a natural (left) $\Fd/\wF$-module structure, given by conjugation in $\Fd$ (on the left). As an $\Fd/\wF$-module, $K$ decomposes into irreducible $\Fd/\wF$-modules as follows:
$$
K=K_0\oplus K_1=\langle a_1+a_2+a_3+a_4+a_5+a_6\rangle \oplus \langle a_1+a_2,a_2+a_3,a_4+a_5,a_5+a_6\rangle
$$
The module $K_0$ has dimension $1$ as an $\F_2$-vector space, and it has trivial $\Fd/\wF$ action. Let us look at the action on $K_1$. We can define an embedding $\psi$ of $\Fd/\wF$ into $\mathrm{GL}_2(\F_4)$: If $\xi\in \F_4$ is a generator of $\F_4$ over $\F_2$ (so it is a root of $X^2+X+1$), it is given by
\begin{equation}\label{isom}
\psi:x  \mapsto
\left(\begin{matrix} \xi & 0\\ 0 & 1\end{matrix}\right)
;
 y  \mapsto   \left(\begin{matrix}0 & 1\\ 1  & 0\end{matrix}\right) 
\end{equation}
Via this embedding of $\Fd/\wF$ into $\mathrm{GL}_2(\F_4)$, $K_1$ can be seen as the $\mathrm{GL}_2(\F_4)$ module $\F_4^2$, after the following homomorphism $\psi'$, which is defined on the basis elements in the following way:
\begin{equation}\label{isom2}
\begin{array}{rrclcrcl}
\psi': &a_1+a_2 & \mapsto & (1,0)&;   & a_4+a_5 & \mapsto & (0,1)\\
&a_2+a_3 & \mapsto & (\xi,0) &;& a_5+a_6 & \mapsto & (0,\xi)
\end{array}
\end{equation}
Thus $K\cong \F_2\oplus \F_4^2$, as an $\Fd/\wF$-module. From now on, whenever we write $\F_2$, we will refer to it as a trivial $\Fd/\wF$-module, and when we write $\F_4^2$, we will mean this vector space with the $\Fd/\wF$-module structure just described. If we look at the conjugate dessins, there is a natural isomorphism from $\wF/\h_1$ and $\wF/\h_2$ to $K$, given by using the homomorphisms $\rho_1$ and $\rho_2$ to $S_{12}$, and then mapping the transpositions as we did for $\h_0$. Via these isomorphisms, $\wF/\h_1$ and $\wF/\h_0$ have the same $\Fd/\wF$-module structure as $\h_0$.

Let us now consider the group $\h_0\cap \h_1\cap\h_2$. The quotient $\Fd/(\h_0\cap \h_1\cap\h_2)$ embeds into $\Fd/\h_0\times \Fd/\h_1\times \Fd/\h_2$, and this embedding is an $\Fd/\wF$-module homomorphism, which we will call $\Phi$. Thus, the latter module is isomorphic to $(\F_2\times \F_4^2)^3$, and $\Fd/(\h_0\cap \h_1\cap\h_2)$ is isomorphic to its image in this module. $\wF$ is generated by $\{x^3,y^2,[x,x^y]\}$ and their conjugates. Therefore, the module we are looking at is generated as a module by $\{x^3,y^2,[x,x^y]\}$. The corresponding images of these elements are as follows:
$$
\begin{array}{c|c|c|c}
\alpha & x^3 & y^2 & \left[x,x^y\right]\\
\hline
\rho_0(\alpha) & a_1+a_2+a_3+a_6 & a_3+a_6 & a_2+a_3+a_5+a_6\\
\rho_1(\alpha)& a_1+a_2+a_3+a_4 & a_3+a_6 & a_1+a_3+a_4+a_6\\
\rho_2(\alpha) & a_1+a_2+a_3+a_5 & a_3+a_6 & a_1+a_2+a_4+a_5\\
\psi'(\rho_0(\alpha)) & (1,(0,1)) & (1,(1,1)) & (0,(\xi,\xi))\\
\psi'(\rho_1(\alpha)) & (1,(0,\xi)) & (1,(1,1)) & (0,(\xi^2,\xi^2))\\
\psi'(\rho_2(\alpha)) & (1,(0,\xi^2)) & (1,(1,1)) &  (0,(1,1))
\end{array}
$$
We claim that the image of $\Fd/(\h_0\cap \h_1\cap\h_2)$ is isomorphic to $\F_2\times \F_4^2\times \F_4^2$, via the following homomorphism:
$$
\begin{array}{rcl}
\Psi:\F_2\times \F_4^2 \times \F_4^2 & \longrightarrow & \F_2^3 \times (\F_4^2)^3\\
(a,(b_0,c_0),(b_1,c_1)) & \longmapsto & (a,a,a,(b_0,c_0),(b_1,c_1),\xi (b_0,c_0)+\xi^2(b_1,c_1))
\end{array}
$$
The image of $\Psi$ is indeed contained in the image of $\Phi$, as one can check on the generators:
$$
\begin{array}{rcl}
\Psi(1,(0,0),(0,0)) & = & b=(1+xx^y+(xx^y)^2)\Phi(x^3) \\
\Psi(0,(1,0),(0,0)) & = & c=(1+x^y+(x^y)^2)(\Phi(x^3)+x^2\Phi(y^2)+\Phi([x,x^y]) \\
\Psi(0,(0,0),(1,0)) & = & (1+x^y+(x^y)^2)((1+y)c+b+\Phi(y^2)) \\
\end{array}
$$
We leave checking that $\Psi$ does induce an isomorphism to the reader. By $\Psi^{-1}$, $\Phi(\h_0/(\h_0\cap  \h_1\cap\h_2))$ is mapped into the submodule of $\F_2\times \F_4^2\times \F_4^2$ such that $a=b_0=c_0=0$, and similarly $\Phi(\h_1/(\h_0\cap  \h_1\cap\h_2))$ is mapped into the submodule such that $a=b_1=c_1=0$.

Thus we have the following diagram of submodules of $\wF/(\h_0\cap \h_1 \cap \h_2)\cong \F_2\times (\F_4^2)^2$, which corresponds to the diagram of normal subgroups on the right:
\begin{center}\hspace{0.5 cm}
\begindc{\commdiag}[3]
\obj(0,30)[b]{$\F_2\times \F_4^2\times \F_4^2$}
\obj(0,15)[b']{$0\times \F_4^2\times \F_4^2$}
\obj(20,7)[e0]{$\F_2\times \F_4^2\times 0$}
\obj(-20,7)[e1]{$\F_2\times 0\times \F_4^2$}
\obj(20,-7)[d0]{$0\times \F_4^2\times 0$}
\obj(-20,-7)[d1]{$0\times 0\times \F_4^2$}
\obj(0,-15)[t]{$\F_2\times 0\times 0$}
\obj(0,-30)[a]{$0$}
\mor{b'}{b}{}[\atleft,\solidline]
\mor{e0}{b}{}[\atleft,\solidline]
\mor{d0}{b'}{}[\atleft,\solidline]
\mor{e1}{b}{}[\atleft,\solidline]
\mor{d1}{b'}{}[\atleft,\solidline]
\mor{d0}{e0}{}[\atleft,\solidline]
\mor{d1}{e1}{}[\atleft,\solidline]
\mor{t}{e0}{}[\atleft,\solidline]
\mor{t}{e1}{}[\atright,\solidline]
\mor{a}{d0}{}[\atleft,\solidline]
\mor{a}{d1}{}[\atright,\solidline]
\mor{a}{t}{}[\atright,\solidline]
\obj(60,30)[b00]{$F$}
\obj(60,15)[b'00]{$\h_0\h_1$}
\obj(40,7)[e000]{$E_0$}
\obj(80,7)[e100]{$E_1$}
\obj(40,-7)[d000]{$\h_0$}
\obj(80,-7)[d100]{$\h_1$}
\obj(60,-15)[t00]{$E_0\cap E_1$}
\obj(60,-30)[a00]{$\h_0\cap \h_1$}
\mor{b'00}{b00}{$2$}[\atleft,\solidline]
\mor{e000}{b00}{$2^4$}[\atleft,\solidline]
\mor{d000}{b'00}{$2^4$}[\atleft,\solidline]
\mor{e100}{b00}{$2^4$}[\atleft,\solidline]
\mor{d100}{b'00}{$2^4$}[\atright,\solidline]
\mor{d000}{e000}{$2$}[\atleft,\solidline]
\mor{d100}{e100}{$2$}[\atright,\solidline]
\mor{t00}{e000}{$2^4$}[\atleft,\solidline]
\mor{t00}{e100}{$2^4$}[\atright,\solidline]
\mor{a00}{d000}{$2^4$}[\atleft,\solidline]
\mor{a00}{d100}{$2^4$}[\atright,\solidline]
\mor{a00}{t00}{$2$}[\atright,\solidline]
\enddc
\end{center}

Corresponding to the submodules $\F_2\times \F_4^2\times 0$ and $\F_2\times 0\times \F_4^2$ we have two normal subgroups $E_0$ and $E_1$ of $\Fd$. The subgroup $\h_2$ corresponds to the module
$$
\{
(a,(b_0,c_0),(b_1,c_1)):a=0,\xi(b_0,c_0)+\xi^2(b_1,c_1)=0
\}
$$
The same as with $\h_0$ and $\h_1$, this module is contained as a codimension 1 submodule in
$$
\{
(a,(b_0,c_0),(b_1,c_1)):\xi(b_0,c_0)+\xi^2(b_1,c_1)=0
\}
$$
This module corresponds to a normal subgroup of $\Fd$, which we will call $E_2$. One can now check that $E_i/\h_i=Z(\Fd/\h_i)$, from which follows that the $E_i$'s are permuted by the Galois action, as we claimed before. Thus, the regular dessin d'enfant $\mE_0$ corresponding to $E_0$ has the same field of moduli as $\h_0$, which is $\Q(\rd)$.

The degree of the dessin is thus $[\Fd:E_0]=[\Fd:\wF][\wF:E_0]=18\cdot 2^4=288$. We can compute the genus of $E_0$ by applying the Riemann-Hurwitz formula: In $\wF/E_0$, which is isomorphic to $\frac{\F_2 \times \F_4^2 \times \F_4^2}{\F_2 \times 0 \times \F_4^2}\cong \F_4^2$, the image of $x^3$ is $(0,1)$ and the image of $y^2$ is $(1,1)$, so $x$ has order 6 and $y$ has order 4. Finally, $(xy)^6$ acts as the permutation $(1,7)(2,8)(3,9)(4,10)(5,11)(6,12)$ on all three dessins $\mH_0$, $\mH_1$ and $\mH_2$, so its image in $\F_2 \times \F_2 \times \F_2 \times \F_4^2\times\F_4^2\times\F_4^2$ equals $(1,1,1,(0,0),(0,0),(0,0))$, and its image in $\wF/E_0$ is 0. Therefore, the order of $xy$ in $\Fd/E_0$ is 6. Thus, the Euler characteristic of the dessin is $[\Fd:E_0](\frac{1}{|x|}+\frac{1}{|y|}+\frac{1}{|xy|}-1)=-120$, and its genus is therefore 61.
\subsection{Explicit equations}
Now we are going to compute equations for $\mE_0$. As an intermediate step, we obtain equations for $\wmF$, the regular dessin corresponding to $\wF$.
\begin{lemma}
The dessin d'enfant $\wmF$ arises from the Belyi pair $(C_{\wmF},\beta_{\wmF})$, where $C_{\wmF}$ is the elliptic curve with equation $X^3+Y^3=2$, and $\beta_{\wmF}$ is given in affine coordinates by $\beta_{\wmF}:(X,Y)\mapsto 1-(1-X^3)^2$. 
\end{lemma}
\begin{proof}
We use the fact that $\wmF$ is the regular cover of $\mF$, and therefore it is the unique regular dessin that has degree 3 over $\mF$. The pair $(C_{\wmF},\beta_{\wmF})$ factors through $\beta_{\mF}$, by the degree 3 map $\pi:(X,Y)\mapsto X$, so that $\beta_{\wmF}=\beta_{\mF}\circ \pi$. It only remains to check that the pair is a regular Belyi pair: indeed, it is the quotient by the group of automorphisms generated by $f_x:(X,Y)\mapsto (\xi X,Y)$ (recall that $\xi=\frac{-1+\sqrt{-3}}{2}$), and $f_y:(X,Y)\mapsto (Y,X)$. One can check that its ramification values are $\{0,1,\infty\}$.

\end{proof}

Using the monodromy, we can draw this dessin, as seen in Figure \ref{Fermat}. Via the isomorphism $\Psi$ in equation (\ref{isom}), $x$ and $y$ are identified with elements in $\mathrm{GL}_2(\F_4)$, and their action on edges corresponds to the left action of $\Fd/\wF$ on itself. Opposite edges with the same matrix label are identified, so the glued polygon becomes homeomorphic to a torus.

\begin{figure}[h!]
\centering
\begin{tikzpicture}[line cap=round,line join=round,>=triangle 45,x=0.7cm,y=0.7cm]
\clip(-10,-8) rectangle (8,8);
\begin{scriptsize}
\draw (0,0)-- (4,0)-- (6,3.46)-- (4,6.93)-- (0,6.93)-- (-2,3.46)-- (0,0);
\draw (4,0)-- (6,-3.46);
\draw (6,-3.46)-- (4,-6.93);
\draw (4,-6.93)-- (0,-6.93);
\draw (0,-6.93)-- (-2,-3.46);
\draw (-2,-3.46)-- (0,0);
\draw (0,0)-- (-2,-3.46);
\draw (-2,-3.46)-- (-6,-3.46);
\draw (-6,-3.46)-- (-8,0);
\draw (-8,0)-- (-6,3.46);
\draw (-6,3.46)-- (-2,3.46);
\draw (-2,3.46)-- (0,0);
\draw (2,3.46)-- ++(-1.0pt,-1.0pt) -- ++(2.0pt,2.0pt) ++(-2.0pt,0) -- ++(2.0pt,-2.0pt) ++ (-1.0pt,-0.4pt) -- ++(0,2.8pt) ++(-1.4pt,-1.4pt)-- ++(2.8pt,0);
\draw (2,-3.46)-- ++(-1.0pt,-1.0pt) -- ++(2.0pt,2.0pt) ++(-2.0pt,0) -- ++(2.0pt,-2.0pt) ++ (-1.0pt,-0.4pt) -- ++(0,2.8pt) ++(-1.4pt,-1.4pt)-- ++(2.8pt,0);
\draw (-4,0)-- ++(-1.0pt,-1.0pt) -- ++(2.0pt,2.0pt) ++(-2.0pt,0) -- ++(2.0pt,-2.0pt) ++ (-1.0pt,-0.4pt) -- ++(0,2.8pt) ++(-1.4pt,-1.4pt)-- ++(2.8pt,0);
\puntoblanco{(5,5.195)}
\puntoblanco{(5,1.73)}
\puntoblanco{(5,-1.73)}
\puntoblanco{(5,-5.195)}
\puntoblanco{(2,6.93)}
\puntoblanco{(2,-6.93)}
\puntoblanco{(2,0)}
\puntoblanco{(-1,5.195)}
\puntoblanco{(-1,-5.195)}
\puntoblanco{(-1,1.73)}
\puntoblanco{(-1,-1.73)}
\puntoblanco{(-4,-3.46)}
\puntoblanco{(-7,1.73)}
\puntoblanco{(-7,-1.73)}
\puntoblanco{(-4,3.46)}
\puntonegro{(0,0)}
\puntonegro{(4,0)}
\puntonegro{(6,3.46)}
\puntonegro{(4,6.93)}
\puntonegro{(0,6.93)}
\puntonegro{(-2,3.46)}
\puntonegro{(6,-3.46)}
\puntonegro{(4,-6.93)}
\puntonegro{(0,-6.93)}
\puntonegro{(-2,-3.46)}
\puntonegro{(-6,-3.46)}
\puntonegro{(-6,3.46)}
\puntonegro{(-8,0)}
\draw (1,0) node[anchor=south] {\scalebox{0.8}{$\mat{1}{0}{0}{1}$}};
\draw (3,0) node[anchor=south] {\scalebox{0.8}{$\mat{0}{1}{1}{0}$}};
\draw (4.5,-0.87) node[anchor=west] {\scalebox{0.8}{$\mat{0}{\xi}{1}{0}$}};
\draw (5.5,-2.6) node[anchor=west] {\scalebox{0.8}{$\mat{1}{0}{0}{\xi}$}};
\draw (5.5,4.33) node[anchor=west] {\scalebox{0.8}{$\mat{\xi^2}{0}{0}{\xi^2}$}};
\draw (4.5,0.87) node[anchor=west] {\scalebox{0.8}{$\mat{0}{\xi^2}{1}{0}$}};
\draw (5.5,2.6) node[anchor=west] {\scalebox{0.8}{$\mat{1}{0}{0}{\xi^2}$}};
\draw (5.5,-4.33) node[anchor=west] {\scalebox{0.8}{$\mat{\xi}{0}{0}{\xi}$}};
\draw (4.5,6.06) node[anchor=west] {\scalebox{0.8}{$\mat{0}{\xi^2}{\xi^2}{0}$}};
\draw (4.5,-6.06) node[anchor=west] {\scalebox{0.8}{$\mat{0}{\xi}{\xi}{0}$}};
\draw (3,6.93) node[anchor=north] {\scalebox{0.8}{$\mat{0}{\xi}{\xi^2}{0}$}};
\draw (1,6.93) node[anchor=south] {\scalebox{0.8}{$\mat{\xi^2}{0}{0}{\xi}$}};
\draw (3,-6.93) node[anchor=south] {\scalebox{0.8}{$\mat{0}{\xi^2}{\xi}{0}$}};
\draw (1,-6.93) node[anchor=north] {\scalebox{0.8}{$\mat{\xi}{0}{0}{\xi^2}$}};
\draw (-0.5,6.06) node[anchor=west] {\scalebox{0.8}{$\mat{\xi}{0}{0}{\xi}$}};
\draw (-0.5,0.87) node[anchor=east] {\scalebox{0.8}{$\mat{\xi}{0}{0}{1}$}};
\draw (-0.5,-6.06) node[anchor=west] {\scalebox{0.8}{$\mat{\xi^2}{0}{0}{\xi^2}$}};
\draw (-1.5,4.33) node[anchor=west] {\scalebox{0.8}{$\mat{0}{\xi}{\xi}{0}$}};
\draw (-1.5,-4.33) node[anchor=west] {\scalebox{0.8}{$\mat{0}{\xi^2}{\xi^2}{0}$}};
\draw (-0.5,-0.87) node[anchor=west] {\scalebox{0.8}{$\mat{\xi^2}{0}{0}{1}$}};
\draw (-1.5,2.6) node[anchor=west] {\scalebox{0.8}{$\mat{0}{1}{\xi}{0}$}};
\draw (-1.5,-2.6) node[anchor=west] {\scalebox{0.8}{$\mat{0}{1}{\xi^2}{0}$}};
\draw (-3,3.46) node[anchor=north] {\scalebox{0.8}{$\mat{0}{\xi^2}{\xi}{0}$}};
\draw (-3,-3.46) node[anchor=south] {\scalebox{0.8}{$\mat{0}{\xi}{\xi^2}{0}$}};
\draw (-5,3.46) node[anchor=south] {\scalebox{0.8}{$\mat{\xi}{0}{0}{\xi^2}$}};
\draw (-5,-3.46) node[anchor=north] {\scalebox{0.8}{$\mat{\xi^2}{0}{0}{\xi}$}};
\draw (-6.5,2.6) node[anchor=east] {\scalebox{0.8}{$\mat{1}{0}{0}{\xi^2}$}};
\draw (-6.5,-2.6) node[anchor=east] {\scalebox{0.8}{$\mat{1}{0}{0}{\xi}$}};
\draw (-7.5,0.87) node[anchor=east] {\scalebox{0.8}{$\mat{0}{\xi^2}{1}{0}$}};
\draw (-7.5,-0.87) node[anchor=east] {\scalebox{0.8}{$\mat{0}{\xi}{1}{0}$}};
\end{scriptsize}
\end{tikzpicture}
\caption{The dessin $\wmF$, where the elements of $\Fd/\wF$ are mapped to $\mathrm{GL}_2(\F_4)$ by the isomorphism (\ref{isom}).}\label{Fermat}
\end{figure}

The next step is to understand the actual coordinates of the vertices of $\wmF$. For this, we will use the automorphism group of $\wmF$, which has two different presentations: one the one hand, it is the full automorphism group of the graph in Figure \ref{Fermat}, and on the other hand, it is the automorphism subgroup of $C_{\wmF}$ generated by $f_x$ and $f_y$ as defined before. We will show that these actions match up as expected.
\begin{lemma}
By the natural homomorphism $\Fd/\wF\to \mathrm{Aut}(C_{\wmF})$, $x$ is mapped to $f_x$ and $y$ is mapped to $f_y$.
\end{lemma}
\begin{proof}

Recall that the edges on a dessin make up the preimage of the segment $[0,1]$. The real segment parametrized by $(t,\sqrt[3]{2-t^3})$, for $t\in [0,1]$ is contained in the preimage of $[0,1]$, so its endpoints, $(0,\sqrt[3]{2})$ and $(1,1)$, are the endpoints of a segment in the dessin. Let us pick the base point on the segment that has first coordinate $1/2$. Its image by $\pi$ is $1/2\in \PP^1$. By the symmetry of the dessin, this edge $e_0$ can be placed anywhere in the picture, so we may choose the one labelled by the identity matrix in Figure \ref{Fermat}. Its endpoints are $(0,\sqrt[3]{2})$ and $(1,1)$.

Recall the definition of the action of $\Fd/\wF$: once a base point is chosen on some edge $e_0$, $x$ acts as the unique automorphism that maps $e_0$ to the edge which results from rotating $e_0$ counterclockwise around its black endpoint. $y$ acts in a similar way, rotating around the white vertex instead of the black one. In $\wmF$, as drawn in Figure \ref{fermatpuntos}, the action of $x$ corresponds to $120^\circ$ rotation around the center black vertex, and $y$ corresponds to $180^\circ$ rotation around the white vertex to the right of this black vertex (the base edge we are choosing is the one connecting both vertices).

The automorphism of $C_{\wmF}$ corresponding to $x$ must fix the central vertex, $(0,\sqrt[3]{2})$, which, after checking through all 18 automorphisms generated by $\{f_x,f_y\}$, means it must be one of $f_x$ or $f_x^2$. If we look at $f_x$, we have that it descends to an automorphism of the Belyi pair $(\PP^1,\beta_{\mF})$, that is, there is an isomorphism $f_x'$ of $\PP^1$ such that $\beta_{\mF}\circ f_x'=\beta_{\mF}$, and also $\pi \circ f_x=f_x' \circ \pi$. Such an automorphism must be unique, and it is in fact $f_x':X\mapsto \xi X$. Similarly, $f_x^2$ descends to $X\mapsto \xi^2 X$. It is easy to check that the monodromy action of $x$ (the loop going around 0) in the pair $(\PP^1,\beta_{\mF})$ maps the base point we've chosen, $1/2$, to $\xi/2$. Therefore, since the automorphism corresponding to $x$ must send the base point $1/2$ to $\xi/2$, this automorphism must be $f_x'$ and not $(f_x')^2$. Since the automorphism of $\wmF$ corresponding to $x$ must descend to the one corresponding to $x$ in the other dessin, we conclude that it is $f_x$ that corresponds to $x$, and not $f_x^2$. Similarly, the automorphism corresponding to rotating counterclockwise around the white vertex marked $(1,1)$ in Figure \ref{fermatpuntos} must fix $(1,1)$. The only element that fixes $(1,1)$ is $f_y$, so $y$ must correspond to $f_y$.
\end{proof}

\begin{figure}[h!]
\centering
\includegraphics[scale=1.3]{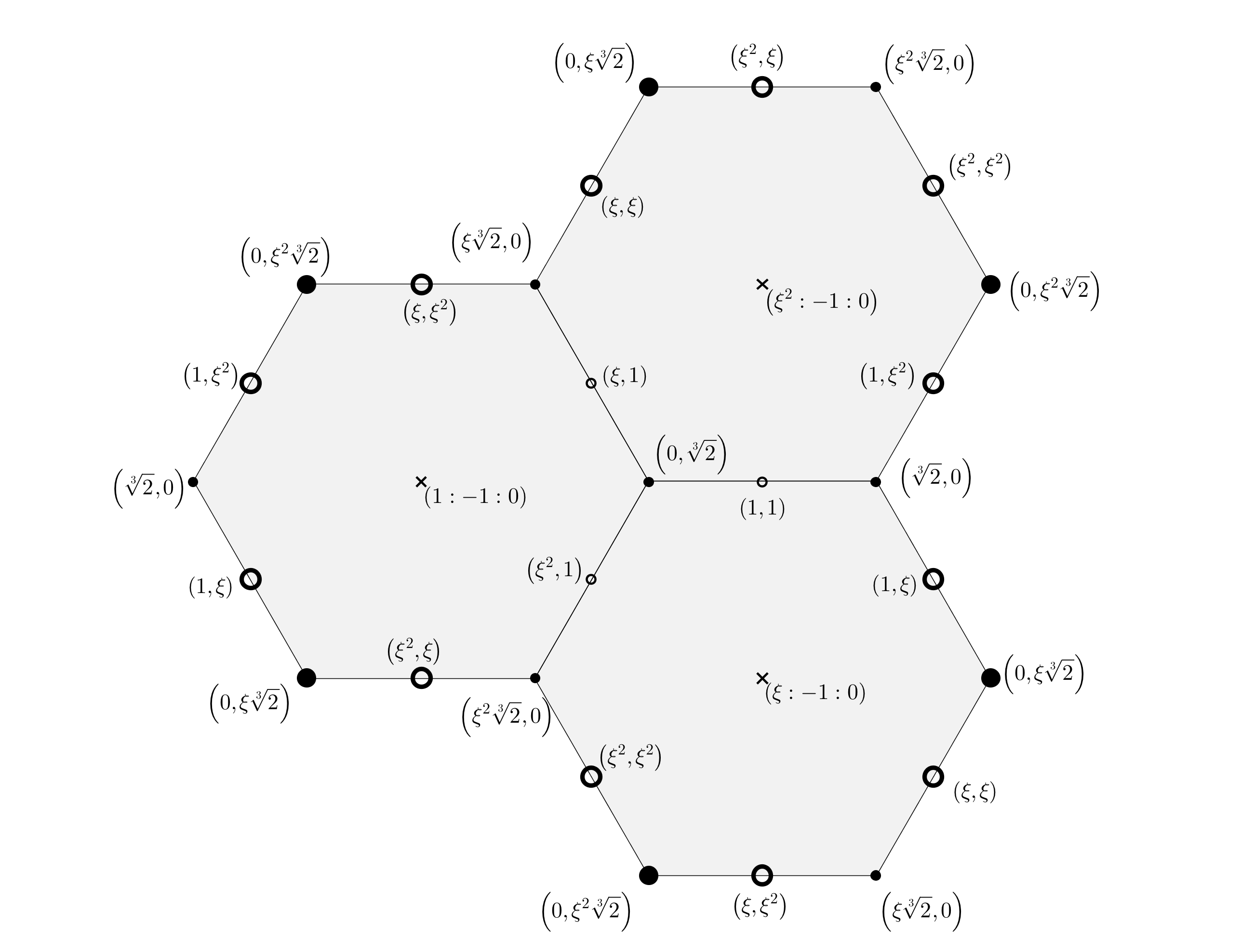}
\caption{The points on the dessin $\wmF$. Bigger vertices are the ones over which $B$ ramifies.}\label{fermatpuntos}
\end{figure}

\begin{lemma}
Up to an automorphism, the coordinates of the vertices of $\wmF$ are as shown in Figure \ref{fermatpuntos}.
\end{lemma}
\begin{proof}

Knowing that $x$ ($120^\circ$ rotation) on the picture corresponds to $f_x$ on the curve, and similarly for $y$, we can identify all the black and white vertices in Figure \ref{fermatpuntos}: for instance, the white vertices connected to $(0,\rd)$  must be $(\xi,1)$ and $(\xi^2,1)$, since these are the ones obtained by applying $f_x$ to $(1,1)$, which in turn is the same as rotating the figure around the black vertex. We can proceed in this manner to label all the black and white vertices.

Only the points at infinity remain to be identified in the figure. If we take the automorphism $yx$, we have that its action is $f_y\circ f_x:(X,Y)\mapsto (Y,\xi X)$, which fixes the point $(\xi:-1:0)$. On the dessin, it is the rotation around $(0,\rd)$ followed by the rotation around $(1,1)$, which fixes the lower right face and interchanges the other two. Therefore, the lower right face corresponds to $(\xi:1:0)$, and the other two are labelled using the $f_x$ action.
\end{proof}

In order to construct the equations for $\mE_0$, which is a degree $2^4$ ramified covering of $\wmF$, we are going to construct a degree 2 ramified covering of $\wmF$, whose regular cover will be $\mE_0$. Intermediate coverings between $\mE_0$ and $\wmF$ are in bijection with subgroups containing $E_0$ and contained in $\wF$, or subgroups of $\wF/E_0$. We are going to look at a covering we will call $\mB$, corresponding to the index 2 subgroup $\{(a,b):a\in \F_4,b\in \F_2\}$ of $\wF/E_0\cong \F_4^2$. This will correspond to a covering of the curve $C_{\wmF}$ unramified outside of the 18 points marked on Figure \ref{fermatpuntos}. $\wF$ is the profinite completion of the fundamental group of a torus with 18 points removed, so it is a free profinite group in 19 generators. The generators correspond to loops going around each of the 18 points, plus two generators, which are the generators of the fundamental group of an unpunctured torus. These 20 generators don't freely generate $\wF$, as it has rank 19. These generators can be seen as elements in $\Fd$ by the inclusion of $\wF$ into $\Fd$, which we understand well. Using this, we can assign an element of $\wF$ to the loop going around each of the vertices, as seen in Table \ref{tablapuntos}. The third column is the image of the element in $\wF/E_0 \cong \F_4^2$ by the isomorphism $\psi'$ in equation (\ref{isom2}).

\begin{table}
\caption{The points on the elliptic curve, an element of $\Fd$ that corresponds to a loop around them and their images in $\wF/E_0\cong \F_4^2$}
\label{tablapuntos}       
\begin{tabular}{ccc|ccc}
\hline\noalign{\smallskip}
Point & Generator & In $\F_4^2$ & Point & Generator & In $\F_4^2$  \\
\noalign{\smallskip}\hline\noalign{\smallskip}
$(0,\rd)$      & $x^3             $& $(0,1)$     & $(1,1)$         & $y^2            $ & $(1,1)$ \\
$(0,\xi\rd)$   & $(x^3)^{x^y}     $& $(0,\xi^2)$ & $(\xi,1)       $& $(y^2)^{x}      $ & $(\xi^2,1)$ \\
$(0,\xi^2\rd)$ & $(x^3)^{(x^y)^2} $& $(0,\xi)  $ & $(\xi^2,1)     $& $(y^2)^{x^2}    $ & $(\xi,1)$\\
$(\rd,0)     $ & $(x^3)^{y}       $& $(1,0)    $ & $(1,\xi)       $& $(y^2)^{xy}     $ & $(1,\xi^2)$ \\
$(\xi\rd,0)  $ & $(x^3)^{yx}      $& $(\xi^2,0)$ & $(\xi,\xi)     $& $(y^2)^{xyx}    $ & $(\xi^2,\xi^2)$ \\
$(\xi^2\rd,0)$ & $(x^3)^{yx^2}    $& $(\xi,0)  $ & $(\xi^2,\xi)   $& $(y^2)^{xyx^2}  $ & $(\xi,\xi^2) $\\
$(1:-1:0)    $ & $((yx)^6)^{x}    $& $(0,0)    $ & $(1,\xi^2)     $& $(y^2)^{x^2y}   $ & $(1,\xi)$\\
$(\xi:-1:0)  $ & $(xy)^6          $& $(0,0)    $ &$(\xi,\xi^2)    $& $(y^2)^{x^2yx}  $ & $(\xi^2,\xi)$\\
$(\xi^2:-1:0)$ & $(yx)^6          $& $(0,0)    $ & $(\xi^2,\xi^2) $& $(y^2)^{x^2yx^2}$ & $(\xi,\xi)$\\
- & $y^{-1}x^{-1}yxyxyx^{-1}$ & $(1,1)$ & - & $y^{-1}x^{-1}y^{-1}x^{-1}y^{-1}xyx$ & $(\xi^2,1) $\\
\noalign{\smallskip}\hline
\end{tabular}
\end{table}

The last two generators in the table are the generators of the fundamental group of an unpunctured torus. The subgroup $B$ contains the elements whose image in $\F_4^2$ has second component in $\F_2$. In Figure \ref{fermatpuntos}, the points over which $\mB$ is ramified are shown as bigger.

Let us look a the structure of the possible degree 2 coverings of the punctured elliptic curve. They are given by index 2 subgroups of $\wF$, which is a free profinite group with 19 generators. The free generators for $\wF$ can be chosen to be all the generators in Table \ref{tablapuntos} except for $x^3$. An index 2 subgroup of $\wF$ will contain the derived subgroup $[\wF,\wF]$, and therefore it will be a subgroup of $\wF/[\wF,\wF]\cong (\Z/2\Z)^{19}$, which in turn can be determined just by the set of generators contained in it. For the 17 generators that correspond to points on the curve, the loop is contained in the corresponding subgroup if and only if the covering map is unramified over that point. Therefore, if one fixes the ramification, there are four degree 2 coverings, depending on whether the last two generators are contained or not in the corresponding subgroup.

The extension of fields of functions corresponding to $\wmF$ is $\QQ(1-(x^3-1)^2)\subset \QQ(X)[Y]/(X^3+Y^3-2)=\kappa$, and $\mB$ corresponds to a degree 2 extension of $\kappa$, so it is generated by the square root of an element $a$ of $\kappa$. The covering will be ramified over the points where $a$ has odd valuation: therefore, $a$ must have odd valuation precisely at the points $(0,\xi\rd),(0,\xi^2\rd),(1,\xi),(\xi,\xi),(\xi^2,\xi),(1,\xi^2),(\xi,\xi^2),(\xi^2,\xi^2)$. Let us see that there can only be four such rational functions up to squares, which is in agreement with the fact that there are only four coverings with that ramification.

\begin{lemma}
On an elliptic curve $E$, there are only 4 functions with even valuation everywhere, up to squares.
\end{lemma}
\begin{proof}
If we let $\QQ(E)^\times$ be the multiplicative group of functions on $E$, and $\Div^0$ be its group of divisors, the fact that $\Pic^0(E)\cong E$ yields a short exact sequence
$$
0\longrightarrow \frac{\QQ(E)^\times}{\QQ^\times} \longrightarrow \Div^0 \longrightarrow \Pic^0(E)\cong E \longrightarrow 0
$$
Tensoring this short exact sequence with $\Z/2\Z$ yields the Tor long exact sequence:
$$
0\longrightarrow \Tor^1(E,\Z/2\Z)\cong(\Z/2\Z)^2 \longrightarrow \frac{\QQ(E)^\times}{(\QQ(E)^\times)^2} \overset{f}{\longrightarrow} \frac{\Div^0}{2\Div^0} \longrightarrow E \longrightarrow 0
$$
The set of functions we are looking for is the kernel of $f$, which by the exactness is isomorphic to $(\Z/2\Z)^2$.
\end{proof}
These functions can be computed by taking one of them, say $\phi$, and then multiplying it by the functions that have even valuation everywhere, to obtain
$$
\phi=\frac{(Y-\xi)(Y-\xi^2)}{(Y-\xi\rd)(Y-\xi^2\rd)};\phi\psi_1;\phi\psi_2;\phi{\psi_1}{\psi_2}
$$Where
$$\psi_1=\frac{X+Y-2}{X+Y-2\xi};\psi_2=\frac{X+Y-2\xi}{X+Y-2\xi^2}
$$In order to decide which of the four functions is the function $a$ such that the extension $\kappa(\sqrt a)/\QQ(1-(X^3-1)^2)$ corresponds to the group $B$, we observe that $B$ satisfies that $xBx^{-1}=B$.
\begin{lemma}
If $B$ is the subgroup of $\Fd$ defined above, we have that $xBx^{-1}=B$.
\end{lemma}
\begin{proof}
By the description of $B$, the isomorphism $\Psi^{-1}$ maps $(B/E_0)$ isomorphically into the submodule $\{(b,c):c\in \F_2\}\subset \F_4^2$. The action of $x$ clearly preserves this submodule, and since the action of $x$ is given by conjugation, we indeed have that $xBx^{-1}=B$.
\end{proof}

Conjugation by $x$ corresponds to applying the transformation $(X,Y)\mapsto (\xi X,Y)$ to the elliptic curve. Out of the four rational functions with odd valuation at the right set of points, only $\phi$ is invariant up to squares by this transformation. Therefore, the field extension corresponding to $B$ is $\kappa(\sqrt \phi)/\QQ(1-(X^3-1)^2)$.

One checks that $B\cap B^{x^y}\cap B^y \cap B^{x^yy}=E_0$. Therefore, the field of functions corresponding to $E_0$ is the composition of the fields obtained from $\kappa(\sqrt \phi)$ when acted on by $1$, $x^y$, $y$ and $x^yy$. We have
$$
\phi_1=\phi=\frac{(Y-\xi)(Y-\xi^2)}{(Y-\xi\rd)(Y-\xi^2\rd)};\phi_2=\phi^{x^y}=\frac{(Y-1)(Y-\xi)}{(Y-\rd)(Y-\xi\rd)};
$$ $$
\phi_3=\phi^y=\frac{(X-\xi)(X-\xi^2)}{(X-\xi\rd)(X-\xi^2\rd)};\phi_4=\phi^{x^yy}=\frac{(X-1)(X-\xi)}{(X-\rd)(X-\xi\rd)}
$$
If we add square roots of all these elements, we obtain the field of functions of the curve $C_{\mE_0}$, as we wished: it is
$$
\QQ(\mE_0)=\frac{\QQ(t)[X,Y,Z_1,Z_2,Z_3,Z_4]}{\left(
t-(1-(X^3-1)^2),X^3+Y^3-2,Z_1^2-\phi_1,Z_2^2-\phi_2,Z_3^2-\phi_3,Z_4^2-\phi_4
\right)}
$$
The field extensions corresponding to $\mE_1$ and $\mE_2$ are obtained by the action of $G_\Q$ on the coefficients of the equations.
\subsection{The field of moduli of the underlying curve}
Let $C=C_{\mE_0}$ be the underlying curve of the dessin $\mE_0$. In this section we prove that, if $\beta_{\mE_0}:C\lra \PP^1$ is the Belyi map corresponding to the dessin $\mE_0$, and $C$ is the underlying curve, then not only the field of moduli of $(C,\beta_{\mE_0})$ is $\Q(\rd)$, but the curve $C$ alone has field of moduli $\Q(\rd)$. By Theorem 4 in \cite{W}, since $C$ is quasiplatonic, then the map $C\to C/\mathrm{Aut}(C)$ is a Belyi
map. Every regular dessin on $C$ must come from a map of the form $C\to C/H$, where $H<\mathrm{Aut}(C)$: therefore, $C$ has a maximal regular dessin, in the sense that there is a unique one of maximum degree.

Let $\sigma\in G_\Q$ such that $\sigma(\rd)\neq \rd$, and suppose that $C^\sigma\cong C$. If this were the case, since $(C,f)\not\cong (C^\sigma,f^\sigma)$, the curve $C$ would have two different regular dessins of degree 288, and therefore it would have a regular dessin of higher degree, correspoding to $C\to C/\mathrm{Aut}(C)$, factoring simultaneously through $f$ and $f^\sigma$.

To prove that this is impossible, we use the results described in \cite{Gi}. Our particular dessin is of type $(6,4,6)$, so it can be seen as the map $C\cong \Gamma\backslash\HH\to \Delta(6,4,6)\backslash\HH\cong \PP^1$, for some $\Gamma \triangleleft \Delta(6,4,6)$, where $\Delta(6,4,6)$ is the triangle group with parameters $(6,4,6)$ acting on $\HH$ in the usual way. The results in \cite{Gi} ensure that the only way that this dessin can be non-maximal is if $\Delta(6,4,6)$ is included with finite index in another triangle group $\Delta'$, and by this inclusion $\Gamma<\Delta'$ is a normal subgroup, so that the dessin $\Gamma\backslash \HH\to  \Delta' \backslash \HH$ is regular.

All finite index embeddings between triangle groups are given in \cite{S}. The only triangle group in which $\Delta(6,4,6)$ can be embedded is $\Delta(6,8,2)$, and this happens in the following way: if $x,y,z$ are the standard generators of $\Delta(6,4,6)$, i.e. so that $x^6=y^4=z^6=xyz=1$, and similarly, $\Delta(6,8,2)$ is generated by $\wt x,\wt y,\wt z$, then the embedding is given by
$$
x=\wt x;y={\wt y}^2;z=\wt z\wt x\wt z
$$
By this embedding, $[\Delta(6,8,2):\Delta(6,4,6)]=2$, and the quotient can be generated by $\wt y$. In order to see that $\Gamma$ is not normal in $\Delta(6,8,2)$, let us look at what happens when conjugating by $\wt y$. First of all, $y^{\wt y}=\wt y^{-1} (\wt y)^2 \wt y=y$. Secondly, since $\wt z=(\wt x \wt y)^{-1}$ has order two, $\wt x\wt y=\wt y^{-1} \wt x^{-1}$, and
$$
x^{\wt y}=\wt y^{-1}\wt x\wt y=\wt y^{-2} \wt x^{-1}=y^{-1}x^{-1}
$$
Now, recall from before that the group $E_0$ contains the element $x^3y^2(x^3y^2)^x(x^3y^2)^{x^2}$. Its image by the conjugation is
$$
\left(x^3y^2(x^3y^2)^x(x^3y^2)^{x^2}\right)^{\wt y}=$$ $$=\left(y^{-1}x^{-1}\right)^3y^2\left((y^{-1}x^{-1})^3y^2\right)^{(y^{-1}x^{-1})}\left((y^{-1}x^{-1})^3y^2\right)^{(y^{-1}x^{-1})^2}
$$
This word has an odd number of $y$'s, therefore it is not contained in the normal subgroup generated by $\{x,y^2\}$. However, this subgroup contains $F$, and therefore $E_0$, so $E_0$ is not normal in $\Delta(6,8,2)$, as we wished to prove. This implies that the maximal regular dessin on $C$ is the one given by $E_0$, so $C\not\cong C^\sigma$, as we wished to prove.

This proves that $C$ is a quasiplatonic curve of genus 61 with field of moduli $\Q(\rd)$. Note that one could use the same technique to prove that $C_{\mathcal H_0}$, the underlying curve of $\mathcal H_0$, also has field of moduli $\Q(\rd)$, since its type is $(6,4,12)$, which is maximal.

\section{Two previous examples}\label{sec:2}
At least two other examples of dessins whose regular cover has non-abelian field of moduli can be found in the literature. The first example is found in \cite{SV}, and it is the family of genus 0 dessins which have passport $(2^2\ 1\ 1,3\ 2\ 1,6)$, shown in Figure \ref{fig:shabat}. It is shown in \cite{SV} that the moduli field of the top dessin in the picture is $\Q(\rd)$, and that they are all Galois conjugate. Associated to the dessins are three subgroups $H_0,H_1,H_2<\Fd$, and their regular covers are associated to the subgroups $\cor_{\Fd}(H_i)$. Since $G_\Q$ acts as automorphisms of $\Fd$, it will permute the groups $\cor_{\Fd}(H_i)$. Therefore, if they are all different, it follows that $\Gal(\Q(\rd,\xi)/\Q)$ acts nontrivially on them.

Suppose that they were not all different: in this case, since the Galois group acts as $S_3$, all of them should be equal. The groups $\Fd/\cor_{\Fd}(H_i)$ are the cartographic groups of the three dessins. These groups are the whole permutation group $S_6$: this can be checked directly. Now, if all the cores are equal, then $S_3\cong\Gal(\Q(\rd,\xi)/\Q)$ must act on the quotient $\Fd/\cor_{\Fd}(H_i)\cong S_6$, and it must permute the subgroups $H_i/\cor_{\Fd}(H_i)$, which are not pairwise conjugate, since they correspond to different dessins. This means that the action of $S_3$ must be by outer automormophisms, but this is impossible, since $\mathrm{Out}(S_6)$ has 2 elements, and thus $S_3$ cannot embed into the group of outer automorphisms. This proves that the regular covers of the dessins are distinct, and thus their fields of moduli are $\Q(\rd)$, $\Q(\xi\rd)$ and $\Q(\xi^2\rd)$.

The other example can be shown to have nonabelian field of moduli by the same reasoning as before: in \cite{R}, the Belyi map called $U_{8,9}^{\mathrm{gen}}$ gives a dessin of genus 0 and passport $(18^39\  1,8\  1^{56},16^4)$. It is shown that there are 35 dessins with this passport, which split into two Galois orbits: one has one element (it is the one corresponding to the Belyi function $U_{8,9}^{\mathrm{gen}}$, and the other has 34 elements. Further, it is proven that the Galois action on the second orbit is by a Galois group isomorphic to $S_{34}$. The cartographic group of these 34 dessins is proven to be $S_{64}$. Doing the same as before, if their regular covers weren't all different, they would be all equal, since the Galois group acts as $S_{34}$, and in particular the action is 2-transitive. If they were all equal, then $S_{34}$ would have to embed into $\mathrm{Out}(S_{64})\cong 1$.

\begin{figure}
\begin{tikzpicture}[line cap=round,line join=round,>=triangle 45,x=1cm,y=0.65cm]
\clip(-1,-2.5) rectangle (7,3);
\draw (1,2)-- (5,2);
\draw (1,0)-- (6,0);
\draw (1,-1)-- (6,-1);
\draw (4,0)-- (4,1);
\draw(4,-1)--(4,-2);
\draw(5,2)--(5.5,2.8);
\draw(5,2)--(5.5,1.2);
\puntonegro{(2,2)}
\puntonegro{(4,2)}
\puntonegro{(5.5,2.8)}
\puntonegro{(5.5,1.2)}
\puntonegro{(1,0)}
\puntonegro{(3,0)}
\puntonegro{(5,0)}
\puntonegro{(1,-1)}
\puntonegro{(3,-1)}
\puntonegro{(5,-1)}
\puntonegro{(4,1)}
\puntonegro{(4,-2)}
\puntoblanco{(1,2)}
\puntoblanco{(3,2)}
\puntoblanco{(5,2)}
\puntoblanco{(2,-1)}
\puntoblanco{(4,-1)}
\puntoblanco{(6,-1)}
\puntoblanco{(2,0)}
\puntoblanco{(4,0)}
\puntoblanco{(6,0)}
\end{tikzpicture}
\caption{The normal cover of these dessins also has nonabelian field of moduli.}
\label{fig:shabat}       
\end{figure}
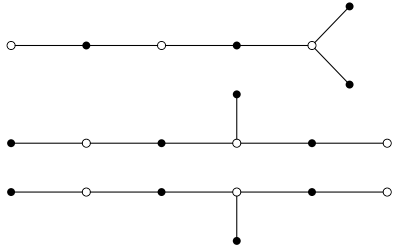


\paragraph*{Acknowledgements} I would like to thank my master's thesis advisor, Andrei Jaikin, for an inmense amount of help and encouragement while I was working on this problem. I would also like to thank the reviewers of this article for many insightful comments which have lead to substantial improvements.


\begin{thebibliography}{}
%
%
\bibitem{CJSW}
M. Conder, G. Jones, M. Streit, J. Wolfart, \textit{Galois actions on regular dessins of small genera}, Rev. Mat. Iberoam. 29, no. 1, pp. 163-181 (2013)

\bibitem{Gi}
E. Girondo, \textit{Multiply quasiplatonic Riemann surfaces}, Experimental Mathematics Volume 12, No.4, pp. 463-475 (2003)

\bibitem{GG}
E. Girondo, G. Gonz\'alez-Diez, \textit{Introduction to Compact Riemann Surfaces and Dessins d'Enfants}, London Mathematical Society Student Texts (2012)

\bibitem{GJ}
G. Gonz\'alez-Diez, A. Jaikin-Zapirain, \textit{The absolute Galois groups acts faithfully on regular dessins and on Beauville surfaces}, Proc. London. Math. Soc. (2015) 111 (4): 775-796

\bibitem{G}
A. Grothendieck, \textit{Esquisse d'un programme}, Geometric Galois Actions. 1. Around Grothendieck's Esquisse d'un programme, ed. L. Schneps and P. Lochak, pp. 5-48, London Math. Soc. Lecture Note Ser. 242, Cambridge University Press (1997)

\bibitem{Gu}
P. Guillot, \textit{An elementary approach to dessins d'enfants and the Grothendieck-Teichm\"uller group}, L'Enseignement Mathématique Vol. 60, Issue 3/4, pp.293-375 (2014)

\bibitem{H}
M. Herrad\'on Cueto, \textit{The field of moduli and fields of definition of dessins d'enfants} (Master's thesis), Universidad Aut\'onoma de Madrid (2014)

\bibitem{J}
M. Jarden, \textit{Normal automorphisms of free pro finite groups}, J. Algebra 62, pp. 118-123 (1980)

\bibitem{R}
D. Roberts, \textit{Chebyshev covers and exceptional number fields}, Preprint

\bibitem{SV}
G. Shabat, V. Voevodsky, \textit{Drawing curves over number fields}, The Grothendieck Festschrift, Progress in Mathematics Volume 88, pp 199-227 (1990)

\bibitem{S}
D. Singerman, \textit{Finitely maximal Fuchsian groups}, J. London Math. Soc. (2) 6, pp. 29-38 (1972)

\bibitem{W}
J. Wolfart, \textit{ABC for polynomials, dessins and uniformization - a survey}, Proceedings der ELAZ-Konferenz 2004, pp. 314-346 Hrsg. W. Schwarz and J. Steuding, Franz Steiner Verlag (2006)


\end{thebibliography}


\end{document}